\documentclass[12pt]{amsart}
\usepackage{amsmath,amssymb,amsbsy,amsfonts,amsthm,latexsym,
                       amsopn,amstext,amsxtra,euscript,amscd}

\usepackage{array}
\usepackage{ifthen}
\usepackage{url}

\newtheorem*{thm}{Theorem}

\newtheorem{lemma}{Lemma}
\newtheorem{proposition}{Proposition}
\newtheorem*{corollary}{Corollary}
\newtheorem*{conjecture}{Conjecture}

\theoremstyle{definition}

\newcommand{\klammern}[4][]%
{\ifthenelse{\equal{#1}{}}{\left#2}{\csname#1\endcsname#2}%
#4\ifthenelse{\equal{#1}{}}{\right#3}{\csname#1\endcsname#3}}
\newcommand{\betrag}[2][]{\klammern[#1]{\lvert}{\rvert}{#2}}

\newcommand{\conj}[1]{^{(#1)}}

\def\QQ{\mathbb Q}
\def\ZZ{\mathbb Z}
\def\RR{\mathbb R}

\begin{document}

\title{Another generalization of a theorem of Baker and Davenport}

\author{Bo He}
\address{Institute of Mathematics \newline \indent
Aba Teachers University \newline
\indent Wenchuan, Sichuan, P. R. China 623000 \newline
\indent and Department of Mathematics \newline \indent
Hubei University for Nationalities\newline
\indent Enshi, Hubei, P. R. China 445000}
\email{bhe@live.cn}

\author{\'Akos Pint\'er}
\address{Institute of Mathematics \newline \indent
 MTA-DE Research Group ``Equations, Functions and Curves'' \newline
\indent Hungarian Academy of Sciences and University of Debrecen \newline
\indent P. O. Box 12, H-4010 Debrecen, Hungary}
\email{apinter@science.unideb.hu}
\thanks{The first and the fourth authors were supported by Natural Science Foundation
of China (Grant No. 11301363), and Sichuan provincial scientific research and innovation team in Universities (Grant No. 14TD0040), and the Natural Science Foundation
of Education Department of Sichuan Province (Grants No. 13ZA0037, 13ZB0036). The second author supported in part by the Hungarian Academy of
Sciences, and OTKA grants K100339, NK101680, NK104208.}

\author{Alain Togb\'e}
\address{Department of Mathematics \newline
\indent Purdue University North Central \newline
\indent 1401 S. U.S. 421, Westville, IN 46391 USA}
\email{atogbe@pnc.edu}

\author{Shichun Yang}
\address{Institute of Mathematics \newline \indent
Aba Teachers University \newline
\indent Wenchuan, Sichuan 623000 \newline
\indent P. R. China}
\email{ysc1020@sina.com}

\date{\today}

\maketitle

\begin{abstract} Dujella and Peth\H{o}, generalizing a result of Baker and Davenport, proved that the set $\{1,3\}$ cannot be extended to a Diophantine quintuple. As a consequence of our main result, it is shown that the Diophantine pair $\{1,b\}$ cannot be extended to a Diophantine quintuple if $b-1$ is a prime.
\end{abstract}

\section{Introduction}\label{sec:1}

A set of $m$ distinct positive integers $\{a_1, \dots,  a_m\}$ is
called a {\it Diophantine $m$-tuple} if $a_i a_j +1$ is a perfect square. Diophantus studied sets of positive rational numbers with the same property,  particularly he found the set of four positive rational numbers $\left\{\frac{1}{16}, \frac{33}{16}, \frac{17}{4}, \frac{105}{16}\right\}$. But the first Diophantine quadruple was found by Fermat observing that the set $\{1, 3, 8, 120\}$ is a Diophantine quadruple. Moreover, Baker and Davenport, in their classical paper \cite{Baker-Davenport:1969}, proved that the set $\{1, 3, 8\}$ cannot be extended to a Diophantine quintuple.

In 1997, Dujella \cite{Dujella:1997-1} obtained that the Diophantine triples of the form $\{k-1, k+1, 4k\}$ cannot be extended to a Diophantine quintuple for $k \geq 2$. The Baker-Davenport's result corresponds to $k=2$. In 1998, Dujella and Peth\H{o} \cite{Dujella-Pethoe:1998} proved that the Diophantine pair $\{1, 3\}$ cannot be extended to a Diophantine quintuple. In 2008, Fujita \cite{Fujita} obtained a more general result by showing that the Diophantine pair $\{k - 1, k + 1\}$ cannot be extended to a Diophantine quintuple for any integer $k\geq 2$. In 2004, Dujella \cite{Dujella:2004} proved that there are only finitely many Diophantine quintuples. A folklore conjecture states that there does not exist a Diophantine quintuple, however, its proof seems well beyond the current techniques. Let
$$
d_{+}= d = a + b + c + 2abc +
2\sqrt{(ab + 1)(ac + 1)(bc + 1)}.
$$
A stronger version of this conjecture is the following
\begin{conjecture}
 If $\{a, b, c, d\}$ is a Diophantine quadruple and $d > \max\{a, b, c\}$, then $d = d_+$.
\end{conjecture}
\noindent We introduce the concept of {\it regular quadruple}.  A Diophantine quadruple $\{a, b, c, d\}$ is regular if and
only if $(a + b - c - d)^2 = 4(ab + 1)(cd + 1)$. Therefore, the
quadruples in the above conjecture are regular.

The aim of this paper is to consider the extensibility of the Diophantine pair $ \{1,b\}$ and to give a new generalization of the mentioned theorems by Baker and Davenport \cite{Baker-Davenport:1969} and Dujella and Peth\H{o} \cite{Dujella-Pethoe:1998}.

We remark that there exists a related result for higher power Diophantine pairs. Bennett \cite{Bennett:2006} proved that the pairs $\{1,b\}$ cannot add a positive integer $c$ such that $b+1,c+1,bc+1$ are both $k$th power, for any integer $k\ge3$. For other results concerning Diophantine $m$-tuples and their generalizations we refer the interested reader to the homepage http://web.math.pmf.unizg.hr/ $\tilde{} $ duje/dn.html.

\vspace{2mm}

If $\{1,b,c\}$ is a Diophantine triple, then there exists positive
integers $r,s$ and $t$ such that
$$ b+1=r^2,\quad c+1=s^2,\quad bc+1=t^2. $$
Thus we have
\begin{equation}\label{eq:st}
    t^2 - bs^2 = 1-b.
\end{equation}
Usually, we cannot get all pairs $(s,t)$ solutions of Pell equation
\eqref{eq:st} without a condition on the parameter $b$. However, when
$b-1$ is a prime power, the solutions $(s,t)$ to equation \eqref{eq:st} are easy to be parameterized by $b$. In fact, we know that there are at most $2^{\omega(l)}$ (here $\omega(l)$ denotes the number of distinct prime factors of $l$) classes solutions to the Pell equation $x^2 - Dy^2 = l$ with $(x,y)=1$. This leads to confirm that all positive solutions to equation \eqref{eq:st} can be expressed as
\begin{equation}\label{eq:express}
    t+s\sqrt{b} = (\pm 1 + \sqrt{b})(r+\sqrt{b})^k = (\pm 1 +
    \sqrt{b})(T_k + U_k\sqrt{b}), \quad k\ge 0,
\end{equation}
where $(T_k,U_k)$ is the $k$th nonnegative  integer solutions to the
Pell equation $T^2 - bU^2=T^2-(r^2-1)U^2=1$. One can show that
\begin{equation}\label{eq:T}
  T_0=1,\quad T_1=r,\quad T_{k+2}=2rT_{k+1}-T_k, \quad k\ge 0
\end{equation}
\begin{equation}\label{eq:U}
   U_0=0,\quad U_1=1,\quad U_{k+2}=2rU_{k+1}-U_k, \quad k\ge 0.
\end{equation}
Thus, we have
\begin{equation}\label{eq:st-form}
(s,t)=(s_k^{(\pm)},
t_k^{(\pm)})= (T_k \pm U_k, \pm T_k + bU_k)
\end{equation}
 and
\begin{equation}\label{eq:c-form}
c_k^{(\pm)}=c=s^2-1 =  T_k^2 \pm 2T_kU_k  + U_k^2 - 1 =  r^2U_k^2
\pm 2T_kU_k.
\end{equation}
We study the extensibility of the Diophantine
triples
$$
\{1,b,c_k^{(\pm)}\}, \quad \mbox{for} \,\, k=1,2,\ldots
$$
and prove the following
\begin{thm}\label{thm:1}
If $\{1,b,c_k^{(\pm)},d\}$ is a Diophantine quadruple with $d>c_k^{(\pm)}$, then $d=c_{k+1}^{(\pm)}$.
\end{thm}
\noindent One can check that the Diophantine quadruple $\{1,b,c_k^{(\pm)},c_{k+1}^{(\pm)}\}$ is regular.

\vspace{2mm}

If $b-1$ is a prime then we recall that the corresponding generalized Pell equation possesses at most two classes of solutions. Thus, the next result is a straightforward consequence of Theorem.

\begin{corollary}\label{cor:1}
If $b-1$ is a prime, then the pair $\{1,b\}$ cannot be extended to a Diophantine quintuple. Moreover, any Diophantine quadruple which contains the pair $\{1, b\}$ is regular.
\end{corollary}

 By Proposition 4.6 of \cite{Fujita-1}, if $c>\max\{100b^7,
20b^8\}$, then $d=d_{+}$. By \eqref{eq:express}, if $k\ge 8$, we get
$$
s \ge \frac{1}{2\sqrt{b}}(\sqrt{b}-1)(2\sqrt{b})^8>2^6b^4
$$
and $c=s^2-1>100b^8$. Consequently, we only need to consider $1\le k\le7$. There are $14$ cases:
\begin{align*}
c_1^{(\pm)} =\!\!  & \quad \!\! r^2 \pm 2r,\\
c_2^{(\pm)}=\!\!  & \quad \!\!4r^4\pm (8r^3-4r),\\
c_3^{(\pm)}=\!\!  & \quad \!\!16r^6 - 8r^4 + r^2 \pm (32r^5 - 32r^3 + 6r),\\
c_4^{(\pm)}=\!\!  & \quad \!\! 64r^8 - 64r^6 + 16r^4 \pm ( 128r^7 - 192r^5 + 80r^3 - 8r),\\
c_5^{(\pm)}=\!\!  & \quad \!\!  256r^{10} - 384r^8 + 176r^6 - 24r^4 + r^2  \pm ( 512r^9 - 1024r^7 + 672r^5 - 160r^3 + 10r),\\
c_6^{(\pm)}=\!\!  & \quad \!\! 1024r^{12} - 2048r^{10} + 1408r^8 - 384r^6 + 36r^4  \\
\!\!  & \pm \!\! ( 2048r^{11} - 5120r^9 + 4608r^7 - 1792r^5 + 280r^3 - 12r),\\
 c_7^{(\pm)}=\!\!  & \quad \!\!4096r^{14} -
10240r^{12} + 9472r^{10} -
3968r^8 + 736r^6 - 48r^4 + r^2\\
\!\!  & \pm \!\!( 8192r^{13} - 24576r^{11} + 28160r^9 - 15360r^7 +
4032r^5 - 448r^3 + 14r).
\end{align*}
Notice that $1<b<c_k^{(\pm)}$ except for $1<c_1^{(-)}<b$. We set $r'=r-1$ in the case $c=c_1^{(-)}$. Then, the triple $\{1,c,b\}$ is $\{1,r'^2-1, r'^2+2r'\}$. It is similar to $\{1,b,c_1^{(+)}\}$.

\section{The system of Pell equations}\label{sec:2}

Let us consider a Diophantine triple $\{1, b, c\}$.  In order to extend this triple to a Diophantine quadruple $\{1,b,c,d\}$, we have to solve the system
\begin{equation}\label{eq:system-d}
\notag    d+1=x^2,\quad bd+1=y^2,\quad cd+1=z^2,
\end{equation}
in integers $x, y, z$. Eliminating $d$, we obtain the following
system of Pell equations
\begin{eqnarray}
\label{eq:system-zx}     z^2-cx^2 &=& 1-c,\\
\label{eq:system-zy}    bz^2-cy^2 &=& b-c,\\
\label{eq:system-yx}     y^2-bx^2 &=& 1-b.
\end{eqnarray}
By \cite[Lemma~1]{Dujella:2001}, if $(z_0, x_0), (z_1, y_1)$ and $(y_2, x_2)$ are the minimal solutions of \eqref{eq:system-zx}, \eqref{eq:system-zy} and \eqref{eq:system-yx},  respectively, then all solutions of \eqref{eq:system-zx}, \eqref{eq:system-zy} and \eqref{eq:system-yx} are given by
\begin{eqnarray}
\label{eq:solutions-zx}     z+x\sqrt{c} &=& (z_0+x_0\sqrt{c})(s+\sqrt{c})^{m},\quad m\ge 0,\\
\label{eq:solutions-zy}  z\sqrt{b}+y\sqrt{c} &=& (z_1\sqrt{b}+y_1\sqrt{c})(t+\sqrt{bc})^{n},\quad n\ge 0,\\
\label{eq:solutions-yx}     y+x\sqrt{b} &=& (y_2+x_2\sqrt{b})(r+\sqrt{b})^{l},\quad l\ge 0.
\end{eqnarray}

Using the above expressions, refer to \cite[Lemma~1]{Dujella:2001}, it
is easy to get the following corresponding sequences to solutions.
All solutions of \eqref{eq:solutions-zx} are given by $z=V_m$ for
some integer $m \ge 0$, where
$$
V_0=z_0,\quad V_1=sz_0+cx_0,\quad V_{m+2}=2sV_{m+1}-V_m,
$$
and all solutions of \eqref{eq:solutions-zy} are given by $z=W_n$
for some integer $n \ge 0$, where
$$
W_0=z_1,\quad W_1=tz_1+cy_1,\quad W_{n+2}=2tW_{n+1}-W_n.
$$
From the sequence $\{W_n\}$ the corresponding solutions of \eqref{eq:solutions-zy} are also given by $y=A_n, \,n\ge 0$, where
\begin{equation}\label{eq:A}
A_0=y_1,\quad A_1=ty_1 + bz_1,\quad A_{n+2}=2tA_{n+1}-A_n.
\end{equation}
From \eqref{eq:solutions-yx}, we conclude that $y=B_l$, for some
integer $l\ge 0$, where
\begin{equation}\label{eq:B}
B_0=y_2,\quad B_1=ry_2+bx_2,\quad B_{l+2}=2rB_{l+1}-B_m.
\end{equation}

Similar to \cite[Lemma~5]{Fujita}, we have the following result.
\begin{lemma}\label{lem:solutions-z0z1}
Assume that $\{1,b,c',c\}$ is not a Diophantine quadruple for any
$c'$ with $0<c'<c_{k-1}^{(\pm)}$. Then, neither $V_{2m+1}=W_{2n}$
nor $V_{2m}=W_{2n+1}$. Moreover,

(1) If $V_{2m}=W_{2n}$ has a solution, then we have $z_0=z_1=\pm 1$.

(2) If $V_{2m+1} = W_{2n+1}$ has a solution, then we have $z_0=\pm t$, $z_1 = \pm s$ and $z_0z_1>0$.
\end{lemma}

Using Lemma~\ref{lem:solutions-z0z1}, we obtain
\begin{lemma}\label{lem:solutions-y2}
Assume that $\{1,b,c',c\}$ is not a Diophantine quadruple for any
$c'$ with $0<c'<c_{k-1}^{(\pm)}$, for  $b\ge 8$. $A_{2n}=B_{2l+1}$ has no solution. Moreover, if $A_{2n}=b_{2l}$ then $y_2=1$. In other cases, we have $y_2=\pm 1$.
\end{lemma}
\begin{proof}
By induction on \eqref{eq:A} and \eqref{eq:B}, we easily get
$$
A_{2n}\equiv y_1 \pmod{b}, \quad A_{2n+1}\equiv ty_1 \pmod{b},
$$
$$
B_{2l}\equiv y_2 \pmod{b}, \quad B_{2l+1}\equiv ry_2 \pmod{b}.
$$
From \cite[Lemma~1]{Dujella:2004} we have
$$
1\le |y_2| \le \sqrt{\frac{(r-1)(b-1)}{2}}<0.8 b^{3/4}<0.5b,
$$
for $b\ge8$.

\vspace{3mm}

If $A_{2n}=B_{2l}$ or $B_{2l+1}$, then by the case (1) of Lemma \ref{lem:solutions-z0z1}, we get $y_1=1$ as $z_1=\pm 1$ and $bz_1^2-cy_1^2=b-c, y_1\ge1$.

$\bullet$ $A_{2n}=B_{2l}$. We have $1\equiv y_2 \pmod b$. Since $|y_2|<b^{3/4}$, then $y_2=1$.

$\bullet$ $A_{2n}=B_{2l+1}$. We have $1\equiv ry_2 \pmod b$, and so $r\equiv r^2y_2 \pmod b$. As $b+1=r^2$, then we get $y_2\equiv r \pmod b$. Since $|y_2|<0.5b$ and $r<0.5b$ for $b\ge 8$, thus $y_2=r$. By \eqref{eq:solutions-yx}, $x_2^2=(y_2^2+b-1)/b = (r^2+b-1)/b=2$. It is impossible.

\vspace{3mm}

 If $A_{2n+1}=B_{2l}$ or $B_{2l+1}$, the case (2) of Lemma \ref{lem:solutions-z0z1} helps to get $y_1=r$ as $z_1=\pm s$ and $bz_1^2-cy_1^2=b-c, y_1\ge1$.

$\bullet$ $A_{2n+1}=B_{2l}$. This implies $rt\equiv y_2 \pmod b$. By \eqref{eq:st-form}, $t=\pm T_{k}+bU_k \equiv \pm T_{k} \pmod b$. One can check that $T_k \equiv 1$ or $r\pmod b$ by $T_0=0, \,T_1=r,\, T_{k+2}=2rT_{k+1}-T_k$. So we have $y_2\equiv \pm r$ or $\pm 1\pmod b$. The case $y_2\equiv \pm r\pmod b$ gives a contradiction like in case (2). Then we have $y_2\equiv \pm1\pmod b$. It results $y_2 = \pm 1$.

$\bullet$ $A_{2n+1}=B_{2l+1}$. In the case, we have $rt\equiv ry_2\pmod b$. With $\gcd (r,b)=1$, we deduce $y_2\equiv t\equiv \pm1\pmod b$ so $y_2 = \pm 1$ again.
\end{proof}

We will determine the integer solutions $(x,y,z)$ of system
\begin{equation}\label{eq:main-pellian-1}
\begin{cases}
    y^2-bx^2 = 1-b, \\
     z^2-cx^2 = 1-c.
\end{cases}
\end{equation}
From the above result, we have to solve the
equation
\begin{equation}\label{eq:plqm}
x=P_l=Q_{m},
\end{equation}
where
\begin{equation}\label{eq:difine-P}
P_{l}=\frac{1}{2\sqrt{b}}\left((y_2+x_2\sqrt{b})\alpha^{l}-(y_2-x_2\sqrt{b})\alpha^{-l}\right),
\end{equation}
\begin{equation}\label{eq:difine-Q}
Q_{m}=\frac{1}{{2\sqrt{c}}}\left((z_0+x_0\sqrt{c})\beta^{m}-(z_0-x_0\sqrt{c})\beta^{-m}\right),
\end{equation}
and $\alpha=r+\sqrt{b}$ and $\beta= s+\sqrt{c}$ are solutions of Pell equations $T^2-bU^2=1$ and $W^2-cV^2=1$, respectively.
Considering all solutions $x=P_l$ of the equation $y^2-bx^2=1-b$, we have
\begin{equation}\label{eq:Pl}
  P_0=x_2, \,\, P_1 = rx_2 + y_2,\,\, P_{l+2} = 2r P_{l+1} - P_{l}.
\end{equation}
All solutions $x=Q_m$ of the equation $z^2-cx^2=1-c$ are determined by
\begin{equation}\label{eq:Qm}
  Q_0=x_0, \,\, Q_1 = sx_0 + z_0,\,\,Q_{m+2} = 2s Q_{m+1} - Q_{m}.
\end{equation}

Referring  to  Lemma~\ref{lem:solutions-y2}, there are two types of fundamental solutions  as follows:

{\bf Type I:  }  $l\equiv m \equiv 0 \pmod{2}$, $x_0=x_2=1,\, z_0=\pm1,\, y_2=1$.

\vspace{3mm}

{\bf Type II:} $ m \equiv 1 \pmod{2}$, $x_0=r, \, x_2=1,\, z_0=\lambda_1 t,\, y_2=\lambda_2,\, \lambda_1,\lambda_2\in \{-1,1\}$.

\section{Gap principle}\label{sec:3}


We will consider the following linear form in logarithms
\begin{equation}\label{eq:def-Lambda}
\Lambda = l\log \alpha - m\log \beta + \log \gamma,
\end{equation}
where
$$\gamma=\frac{\sqrt{c}(y_2 + x_2\sqrt{b})}{\sqrt{b}(z_0+x_0\sqrt{c})}.$$

\begin{lemma}\label{lem:lambda} If $P_l=Q_m$ has solution $(l,m)$ with $m\ge 1$, then
$0<\Lambda <  \beta^{-2m}$.
\end{lemma}
\begin{proof}
Put
$$
E=\frac{y_2+x_2\sqrt{b}}{\sqrt{b}}\alpha^{l}\quad \mbox{and} \quad
F=\frac{z_0+x_0\sqrt{c}}{\sqrt{c}}\beta^{m}.
$$
It is clear that $E,\,F>1$ if $m\ge 1$. Then equation $P_l=Q_m$ becomes
$$
E + \frac{b-1}{b}E^{-1} = F + \frac{c-1}{c}F^{-1}.
$$
Since $c>b\ge 8$, we have $\frac{c-1}{c}>\frac{b-1}{b}$.
It follows that
\begin{equation}\label{eq:P-Q}
E + \frac{b-1}{b}E^{-1}>F + \frac{b-1}{b}F^{-1},
\end{equation}
and hence
$$
(E-F)(EF-\frac{b-1}{b}) >0.
$$
So we get $E>F$. Moreover, by \eqref{eq:P-Q}
we have
$$
0<E-F<\frac{c-1}{c}E^{-1}<E^{-1}<F^{-1}.
$$
Therefore, we have $\Lambda>0$ and
$$
\Lambda = \log
\frac{E}{F}=\log\left(1+\frac{E-F}{F}\right)<\frac{E-F}{F}<F^{-2}.
$$
Considering all cases in types I, II, we have $\Lambda<\beta^{-2m}$.
\end{proof}

Put
\begin{equation}\label{eq:def-lambda}
\lambda = \begin{cases}
0,\quad \mbox{if the solution $(l,m)$ is of Type I,} \\
1, \quad \mbox{if the solution $(l,m)$ is of Type II with $\lambda_1=1$,}\\
-1, \quad \mbox{if the solution $(l,m)$ is of Type II with $\lambda_1=-1$}.
\end{cases}
\end{equation}
We obtain the following result.
\begin{lemma}\label{lem:close} If $P_l=Q_m$ has a solution $(l,m)$ with $m\ge 1$, then for $r\ge 1000$, we have
\begin{equation}\label{eq:close}
|(l-\lambda)\log\alpha -m \log\beta | < \frac{2}{\sqrt{b}}.
\end{equation}
\end{lemma}
\begin{proof}
By the definition of $\Lambda$ in \eqref{eq:def-Lambda}, we have
\begin{equation}\label{eq:ineq}
|(l-\lambda)\log\alpha -m \log\beta | = \left|\Lambda - \log \gamma-\lambda \log \alpha\right|\le |\Lambda| +  \left|\log \gamma + \lambda \log \alpha\right|.
\end{equation}

One can easily get $0<\Lambda < \frac{1}{4c}$. To estimate inequality \eqref{eq:ineq}, we will consider three cases according to the values of $\lambda$.

\vspace{2mm}

{\bf Case I:} If $\lambda =0$, then the solution $(l,m)$ is of Type I. It implies $ x_0=x_2=1,\, z_0=\pm1,\, y_2=1$. We have
$$
\gamma = \frac{\sqrt{c}(y_2 + x_2\sqrt{b})}{\sqrt{b}(z_0+x_0\sqrt{c})} = \frac{\sqrt{c}(1 + \sqrt{b})}{\sqrt{b}(\pm1 + \sqrt{c})}=\frac{1+\sqrt{b}}{\sqrt{b}}\cdot\frac{\sqrt{c}}{\sqrt{c}\pm1}>1.
$$
This and $c\ge r^2+2r $ give
$$
0<\log \gamma \le \log\left(1+\frac{1}{\sqrt{b}}\right) + \log\left(1+\frac{1}{\sqrt{c}-1}\right)<\frac{1}{\sqrt{b}}+\frac{1}{\sqrt{c}-1}<\frac{2}{\sqrt{b}}.
$$
It implies
$$
|\Lambda-\log \gamma|<\left|\frac{1}{4c}-\frac{2}{\sqrt{b}}\right|<\frac{2}{\sqrt{b}}.
$$
This and \eqref{eq:ineq} prove the lemma in the case $\lambda=0$.

\vspace{2mm}

{\bf Case II:} If $\lambda = 1\; (=\lambda_1)$, then the solution $(l,m)$ is of Type II, with $x_0=r, x_2=1,\, z_0=\lambda_1 t=t,\, y_2=\lambda_2,\, \lambda_2\in \{-1,1\}$. We get
$$
 \gamma = \frac{\sqrt{c}(\lambda_2 + \sqrt{b})}{\sqrt{b}(t+r\sqrt{c})}.
$$
As
$$
\alpha^{\lambda}\gamma -1 = \frac{\sqrt{c}(\lambda_2 + \sqrt{b})(r+\sqrt{b})}{\sqrt{b}(t+r\sqrt{c})}-1
=\frac{\lambda_2\sqrt{c}(r+\sqrt{b})-\frac{\sqrt{b}}{t+\sqrt{bc}}}{\sqrt{b}(t+r\sqrt{c})},
$$
we have
$$
|\alpha^{\lambda}\gamma -1|< \frac{\sqrt{c}(r+\sqrt{b})+0.01}{\sqrt{b}(t+r\sqrt{c})}.
$$
As $\sqrt{c}(r+\sqrt{b})-(t+r\sqrt{c})=\sqrt{bc}-t\leq 0$, we see that
$$|\alpha^{\lambda}\gamma -1|<\frac{1.01}{\sqrt{b}}<0.011.$$
It results that
$$
|\log(\alpha^{\lambda}\gamma)|=|\log(1+(\alpha^{\lambda}\gamma-1))|<\frac{1.02}{\sqrt{b}}.
$$
Combining the above inequality and \eqref{eq:ineq}, we obtain
$$
|(l-\lambda)\log\alpha -m \log\beta |<\frac{1}{4c}+ \frac{1.02}{\sqrt{b}}<\frac{1.03}{\sqrt{b}}.
$$

{\bf Case III:} If $\lambda =-1\; (=\lambda_1)$, then the solution $(l,m)$ is of Type II, with $x_0=r, x_2=1,\, z_0=-t,\, y_2=\lambda_2,\, \lambda_2\in \{-1,1\}$. We have
$$
 \gamma = \frac{\sqrt{c}(\lambda_2 + \sqrt{b})}{\sqrt{b}(-t+r\sqrt{c})},
$$
and so
$$
\alpha^{\lambda}\gamma = \alpha^{-1}\gamma = \frac{\sqrt{c}(\lambda_2 + \sqrt{b})(t+r\sqrt{c})}{\sqrt{b}(c-1)(r+\sqrt{b})}.
$$
We get
$$
\alpha^{\lambda}\gamma-1= \frac{t\sqrt{bc}-bc+\sqrt{b}(r+\sqrt{b})+\lambda_2\sqrt{c}(t+r\sqrt{c})}{\sqrt{b}(c-1)(r+\sqrt{b})}.
$$
With $t\sqrt{bc}-bc=\sqrt{bc}(t-\sqrt{bc})=\frac{\sqrt{bc}}{t+\sqrt{bc}}<1/2$, we obtain
$$
|\alpha^{\lambda}\gamma-1|<\frac{1}{2\sqrt{b}(c-1)(r+\sqrt{b})}+\frac{1}{c-1}+\frac{\sqrt{c}(t+r\sqrt{c})}{\sqrt{b}(c-1)(r+\sqrt{b})}.
$$
From $$
\frac{\sqrt{c}(t+r\sqrt{c})}{(c-1)(r+\sqrt{b})}-1= \frac{r+\sqrt{b}+\frac{\sqrt{c}}{t+\sqrt{bc}}}{(c-1)(r+\sqrt{b})}< \frac{r+\sqrt{b}+0.01}{(c-1)(r+\sqrt{b})}<\frac{1.01}{c-1}<0.02,
$$
we get
$$
|\alpha^{\lambda}\gamma-1|<\frac{1}{2\sqrt{b}(c-1)(r+\sqrt{b})}+\frac{1}{c-1}+\frac{1.02}{\sqrt{b}}<\frac{1.04}{\sqrt{b}}.
$$
Hence, we deduce
$$
|(l-\lambda)\log\alpha -m \log\beta |<\frac{1}{4c}+ \frac{1.04}{\sqrt{b}}<\frac{1.05}{\sqrt{b}}.
$$
This completes the proof of the lemma.
\end{proof}

Put
\begin{equation}\label{eq:Det}
\Delta = l-\lambda -km.
\end{equation}
\begin{lemma}\label{lem:deltano}
If $P_l = Q_m$ has a solution $(l,m)$ with $m>1$, then $\Delta\neq 0$.
\end{lemma}
\begin{proof}
Assume that $\Delta = l- \lambda -km =0$.  From \eqref{eq:Pl}, by induction one gets
\begin{equation}\label{eq:Plk}
  P_0=x_2, \,\, P_k = x_2T_k + y_2U_k,\,\, P_{mk+2k} = 2T_k P_{mk+k} - P_{mk}.
\end{equation}
\vspace{5mm}
{\bf Case I:} $\lambda =0$.  This is Type I with  $x_0=x_2=1,\, z_0=\pm1,\, y_2=1$. We have $l=km$. By \eqref{eq:Plk} and \eqref{eq:Qm}, we obtain
\begin{equation}\label{eq:Plk-1}
  P_0=1, \,\, P_k = T_k + U_k,\,\, P_{mk+2k} = 2T_k P_{mk+k} - P_{mk}
\end{equation}
and
\begin{equation}\label{eq:Qm-1}
  Q_0=1, \,\, Q_1 = s \pm1,\,\,Q_{m+2} = 2s Q_{m+1} - Q_{m}.
\end{equation}
Recall that  $s= s_k^{(\pm)} = T_k \pm U_k$.

When $s=s_{k}^{(-)}$, from $P_k>Q_1$ and $2T_k > 2s$, we have $P_{km}>Q_m$, for $m\ge 1$.

When $s=s_k^{(+)}$ and $Q_1 = s+1$, by $P_k < Q_1$ and $2T_k < 2s$, we obtain  $P_{km} < Q_m$, for $m\ge 1$.

When $s=s_k^{(+)}$ and $Q_1 = s-1$, we get $P_k < Q_1$. From $2T_k <2s$ and
$$P_{2k} = 2T_{k}(T_k+U_k) -1 < 2(T_k+U_k)(T_k + U_k -1) -1 = Q_2$$
we have $P_{mk} < Q_m$, for $m\ge 2$.

Thus, we obtain $P_{km} \neq Q_m$ in Type I. This contradicts the fact that $l=km$.

\vspace{5mm}

{\bf Case II:} $\lambda =1$. We are in Type II with $\lambda_1=1$, $x_0=r, \, x_2=1,\, z_0= t,\, y_2=\lambda_2 = \pm1$. If $\Delta =0$, then $l= km+1$.  By \eqref{eq:Plk-1} and \eqref{eq:Qm}, we have
\begin{equation}\label{eq:Plk-2}
  P_0=1, \,\, P_k = T_k \pm U_k,\,\, P_{mk+2k} = 2T_k P_{mk+k} - P_{mk}
\end{equation}
and
\begin{equation}\label{eq:Qm-2}
  Q_0=r, \,\, Q_1 = rs +t ,\,\,Q_{m+2} = 2s Q_{m+1} - Q_{m}.
\end{equation}
If $m=0$, then $l=1$. $P_1 = rx_2 + y_2 = r\pm 1 \neq r= Q_0$.  If $m =1$, then $l= k+1$. $P_{k+1} = r(T_k\pm U_k) \pm T_k +bU_k$. $Q_1 = rs+t = rs_{k}^{(\pm)} + t_{k}^{(\pm)} = r(T_k \pm U_k) + (\pm T_k +bU_k)$.

When $s= s_k^{(\lambda_2)}$, then $P_{k+1} = Q_1$. But by induction $2T_{k}\neq 2s$  provides  $P_{km+1}  \neq Q_m$, for $m\ge 2$.

When $s=s_k^{(+)}$ and $\lambda_2 =-1$. $P_{k+1}<Q_1$ and $2T_k < s$ imply $P_{km+1} < Q_m$.

When $s= s_k^{(-)}$ and $\lambda_2 =1$. $P_{k+1}<Q_1$ and $2T_k > s$ imply $P_{km+1} >Q_m$.

Therefore, $P_{km+1} \neq Q_{m}$. It contradicts our assumption $l=km+1$.

\vspace{5mm}

{\bf Case III:} $\lambda =-1$. It is of Type II with $\lambda_1=-1$, $x_0=r, \, x_2=1,\, z_0=- t,\, y_2=\lambda_2 = \pm1$.  The proof is similar to that in Case II.
\end{proof}

 \begin{lemma}\label{lem:m-Delt} If $P_l=Q_m$ has a solution $(l,m)$ with $m\ge 1$, then for $r\ge 1000$, we have
$$
m>0.98 |\Delta|\sqrt{b}\cdot \log \alpha.
$$
\end{lemma}
\begin{proof}
From Lemma \ref{lem:close}, we have $|(l-\lambda)\log\alpha - m\log\beta|<\frac{2}{\sqrt{b}}$ and then
$$
\left|\frac{l-\lambda}{m}-\frac{\log\beta}{\log\alpha} \right|<\frac{2}{m\sqrt{b}\cdot \log\alpha}.
$$
Thus, we have
\begin{equation}\label{eq:pre-k}
\frac{|\Delta|}{m}=\left|\frac{l-\lambda-km}{m}\right|<\left|\frac{\log\beta}{\log\alpha}-k\right|+\frac{2}{m\sqrt{b}\cdot \log\alpha}.
\end{equation}
Moreover, we get
\begin{equation}\label{eq:albt}
\left|\frac{\log\beta}{\log\alpha}-k\right|=\left|\frac{\log\beta-\log\alpha^k}{\log\alpha}\right|=\left|\frac{\log(\beta/\alpha^k)}{\log\alpha}\right|
=\left|\frac{\log(1+(\beta-\alpha^k)/\alpha^k)}{\log\alpha}\right|.
\end{equation}
As $\alpha^k = (r+\sqrt{b})^k=T_k+U_k\sqrt{b}$ and $\beta=s+\sqrt{c}$, with $s=s_k^{(\pm)}= T_k\pm U_k$,  we have
$$
\left|\frac{\beta-\alpha^k}{\alpha^k}\right|= \left|\frac{s+\sqrt{c}-(r+\sqrt{b})^k}{(r+\sqrt{b})^k}\right|=\left|\frac{2s+\frac{1}{s+\sqrt{c}}-(2T_k+\frac{1}{T_k+U_k\sqrt{b}})}{T_{k}+U_k\sqrt{b}}\right|
$$
$$
=\left|\frac{\pm 2U_k +\frac{1}{s+\sqrt{c}}-\frac{1}{T_k+U_k\sqrt{b}}}{T_{k}+U_k\sqrt{b}}\right|<\frac{ 2U_k+0.01}{T_{k}+U_k\sqrt{b}}<\frac{ 2U_k+0.01}{2U_k\sqrt{b}}<\frac{1.001}{\sqrt{b}}.
$$
We deduce that
\begin{equation}
\left|\log(1+(\beta-\alpha^k)/\alpha^k)\right|<1.01\left|\frac{\beta-\alpha^k}{\alpha^k}\right|<\frac{1.012}{\sqrt{b}}.
\end{equation}
Combining this,  \eqref{eq:pre-k} and \eqref{eq:albt}, we obtain
$$
\frac{|\Delta|}{m}<\frac{1.02}{\sqrt{b}\cdot\log \alpha} + \frac{2}{m\sqrt{b}\cdot \log\alpha}=\frac{1.012+\frac{2}{m}}{\sqrt{b}\cdot \log \alpha}.
$$
Therefore, it results
$$
1.012m+2> |\Delta|\sqrt{b}\cdot \log \alpha.
$$
This implies
$$
m>0.98 |\Delta|\sqrt{b}\cdot \log \alpha ,
$$
and the proof is completed.
\end{proof}

Moreover, we have the following result.
\begin{lemma}\label{lem:mr}
If $P_l=Q_m$ has a solution $(l,m)$, then $m\equiv 0,\pm1\pmod {r}$.
\end{lemma}
\begin{proof}
By induction, we have
$$
P_{2l}\equiv x_2\pmod{r},\quad P_{2l+1} \equiv y_2\pmod{r}.
$$
From $t^2-bs^2=1-b$, $s=s_k^{(\pm)}=T_{k}+\lambda_3U_k$ and the sequences of $\{T_k\},\,\{U_k\}$ given by \eqref{eq:T} and \eqref{eq:U}, we have
$s\equiv \pm1 \pmod{r}$. Let $s\equiv \pm \lambda_3 \pmod{r},\, \lambda_3\in\{-1,1\}$. From \eqref{eq:Qm}, we get
$$
Q_{m}\equiv \lambda_3^{m-1}(\lambda_3x_0+ z_0m) \pmod{r}.
$$
We will consider two cases.
\vspace{5mm}

 {\bf Type I: } with $ l\equiv m \equiv 0\pmod{2}$ and $x_0=1,z_0 = \pm1$. Then, we have
$$P_{l}\equiv x_2=1\pmod {r},\quad Q_{m}\equiv x_0+\lambda_3z_0m = 1 \pm m \pmod{r}, $$
and $m\equiv 0 \pmod r$.
\vspace{5mm}

{\bf Type II: } with $m\equiv 1\pmod 2$ and $x_0=r, z_0 = \pm t, y_2=\pm1$. Thus, we get
 $$P_{l}\equiv \pm 1\pmod {r},\quad Q_{m}\equiv \lambda_3 r +z_0m \equiv z_0m  \pmod{r}. $$
 The fact  that $t=\pm T_k +bU_k \equiv \pm T_k -U_k \equiv \pm s \pmod{r}$ helps to obtain $Q_m\equiv z_0m \equiv \pm m \pmod{r}$. Therefore, we deduce $m\equiv \pm1 \pmod r$.
\end{proof}

\section{Linear forms in two logarithms}\label{sec:4}

Now, we recall the following result due to Laurent  (see ~\cite{Laurent:2008}, Corollary 2) on linear forms in two logarithms. For any non-zero
algebraic number $\alpha$ of degree $d$ over $\QQ$, whose minimal
polynomial over $\ZZ$ is $a\prod_{j=1}^d \left(X-\alpha\conj j
\right)$, we denote by
\[
h(\alpha) = \frac{1}{d} \left( \log|a| + \sum_{j=1}^d
\log\max\left(1, \betrag{\alpha\conj j}\right)\right)
\]
its absolute logarithmic height.
\begin{lemma}\label{lem:LMN}
Let $\alpha_1$ and $\alpha_2$ be  multiplicatively independent algebraic numbers and $\alpha_1, \alpha_2, \log \alpha_1, \log \alpha_2 $ are
real and positive, $b_1$ and $b_2\in \ZZ$ and
$$\Lambda= b_2\log\alpha_2 - b_1\log\alpha_1.$$
Let $D:=[\QQ(\alpha_1,\alpha_2):\QQ]$, for $i=1,2$ let
$$h_i\ge\max\left\{h(\alpha_i), \frac{\betrag{\log\alpha_i}}{D},
\frac1{D}\right\}$$
 and
$$b'\ge\frac{\betrag{b_1}}{D\,h_2} + \frac{\betrag{b_2}}{D\,h_1}.$$
If $\betrag{\Lambda}\neq 0$, then we have
$$\log\betrag{\Lambda} \ge -17.9\cdot D^4 \left(\max \left\{\log
b'+0.38, \frac{30}{D},\frac12\right\}\right)^2h_1 h_2.$$
\end{lemma}

\vspace{0.3 cm}

As stated in Section~\ref{sec:1}, we only need to consider the extensibility of the triples $\{1,b,c_1^{(+)}\}$ and $\{1,b,c_{k}^{(\pm)}\}$, for $2\le k\le 7$. Assume that $r \ge1000$ and  $P_l=Q_m$ has a solution
$(l,m)$ with $l, m\ge1$.
We have
$$
\Lambda = l\log \alpha - m\log \beta + \log \gamma.
$$
In order to apply Lemma~\ref{lem:LMN}, we put $\Delta = l-\lambda -km $ and rewrite $\Lambda$ into the form
\begin{equation}
\Lambda =
\log\left(\alpha^{\Delta+\lambda}\gamma\right)-m\log\left(\frac{\beta}{\alpha^{k}}\right).
\end{equation}
Hence, we take
$$D=4,\quad b_1=m,\quad b_2=1,\quad \alpha_1=\frac{\beta}{\alpha^k},\quad \alpha_2=
\alpha^{\Delta+\lambda}\gamma .$$
One can check $\alpha_1, \alpha_2, \log \alpha_1$ and $\log \alpha_2$ are real and positive. Otherwise we work on $\Lambda=m\log ({\alpha^k}/\beta) - \log (\alpha^{-\Delta-\lambda}\gamma^{-1}) $.

 Since $\alpha=r+\sqrt{b},\,\beta=s+\sqrt{c}$, then
 $\alpha_1=\frac{(r+\sqrt{b})^k}{s+\sqrt{c}}=\frac{T_{k}+U_k\sqrt{b}}{s+\sqrt{c}}$ is
a zero of the polynomial
$$ X^4+4sT_kX^3-4(T_k^2+s^2+1)X^2+4sT_kX+1. $$
 The absolute values of its conjugates whose greater than $1$ are $\alpha^k \beta$ and
 $$
 \begin{cases}
{\beta}/{\alpha^k},\quad \mbox{if }\,\, s=s_k^{(+)},\\
 {\alpha^k}/{\beta},\quad \mbox{if }\,\, s=s_k^{(-)}.
 \end{cases}
 $$
Hence
$$
h(\alpha_1)=\frac{1}{4}\max\left\{\left(\log({\alpha^k}{\beta})\pm \log({\alpha^k}/{\beta})\right)\right\}=
\frac{k}{2}\log\alpha\quad \mbox{or}\quad \frac{1}{2}\log \beta.
$$
Also, it is easy to see that
$h\left(\alpha^{\Delta+\lambda}\right)=\frac{1}{2}|\Delta+\lambda|\cdot \log \alpha$ and
$$
h(\gamma)=h\left( \frac{\sqrt{c}(y_2+x_2\sqrt{b})}{\sqrt{b}(z_0+x_0\sqrt{c})}\right) \le h\left(\frac{y_2+x_2\sqrt{b}}{\sqrt{b}}\right)+ h\left(\frac{z_0+x_0\sqrt{c}}{\sqrt{c}}\right)
$$
$$
\le \frac{1}{2}\log(b+\sqrt{b})+ \frac{1}{2}\log (rc+t\sqrt{c})<\frac{1}{2}\log(4rbc)< \log \alpha + \frac{1}{2}\log \beta.
$$
 Thus, we have
$$
h(\alpha_2)=h\left(\alpha^{\Delta+\lambda}\gamma\right)\le h\left(\alpha^{\Delta+\lambda}\right) +h(\gamma)\leq \frac{1}{2}
(|\Delta+\lambda|+2)\log\alpha + \frac{1}{2}\log \beta.
$$
As $r\ge 1000$, we obtain $\frac{|\alpha^k-\beta|}{\alpha^k}<\frac{1}{\sqrt{b}}<0.001$. This helps to get $|\log \alpha^k -\log\beta|<0.002$ and then
$$
h_1= \frac{k}{2} \log \alpha +0.01>h(\alpha_1),\quad  h_2 = \frac{1}{2}(|\Delta+\lambda|+2+k)\log\alpha + 0.01 <h(\alpha_2).
$$
 Moreover, we have $\frac{|b_2|}{Dh_1}=\frac{1}{2k\log \alpha+0.04}<0.07$.
So we put
\begin{equation}\label{eq:define-b'}
b'=\frac{m}{2(|\Delta+\lambda|+2+k)\log \alpha+0.04}+0.07.
\end{equation}

We will bound $b'$. If $\log b'+0.38 \le 30/D = 7.5$, then
$$
b'<1237.
$$
Otherwise, by
Lemma~\ref{lem:LMN} we obtain
\begin{equation}\label{eq:log-lambda-1}
    \log\betrag{\Lambda} \ge -17.9\cdot 4^4 \left(\log
b'+0.14\right)^2h_1h_2.
\end{equation}
On the other hand, from Lemma~\ref{lem:lambda} we get $\log|\Lambda|<-2m\log \beta$.
Thus we have
$$
m\log \beta < 17.9\cdot 128 \left(\log
b'+0.38\right)^2h_1h_2.
$$
Since $\log\beta>\log\alpha^k -0.002 >2h_1-0.03$, then
$$
m < 1.01\cdot 17.9\cdot 64 \left(\log
b'+0.38\right)^2h_2.
$$
It follows that
$$
b'-0.07=\frac{m}{4h_2} < 289.27 \left(\log
b'+0.38\right)^2.
$$
We calculate that $b'<33791$ ($>1237$).
Therefore,  we get the following.

\begin{lemma}\label{lem:m-upper}  For a triple $\{1,b,c_{k}^{(\pm)}\},\,(1\le k\le 7)$, if $P_l=Q_m$ has a solution $(l,m)$ with $m\ge 1$ and $r\ge1000$, then we have
$$
m<67582(|\Delta+\lambda|+2+k)\log \alpha+1352.
$$
\end{lemma}
The following result comes from Lemmas~\ref{lem:m-Delt} and \ref{lem:m-upper}.

\begin{proposition}\label{pro:1}
 For a triple $\{1,b,c_{k}^{(\pm)}\}, \,(1\le k\le 7)$, if
$r>68962(k+4)+115$, then the equation $P_l=Q_m$ has no solution $(l,m)$ satisfying $m>1$.
\end{proposition}
\begin{proof}
 Assume that $r\ge 1000$. Since $\Delta \neq 0$, then $|\Delta|\ge 1$. By Lemma~\ref{lem:m-Delt} and Lemma \ref{lem:m-upper}, we have
$$
0.98|\Delta|\sqrt{b}\log\alpha < 67582(|\Delta+\lambda|+2+k)\log \alpha+1352.
$$
This implies
$$
r-1<\sqrt{b} < \frac{67582(|\Delta+\lambda|+2+k)}{0.98|\Delta|} + \frac{1866.1}{0.98|\Delta|\log \alpha}
$$
$$
<68962(k+4)+114.
$$
 Therefore, this completes the proof of  Proposition \ref{pro:1}.
\end{proof}

\section{Proof of Theorem~\ref{thm:1}}\label{sec:5}

In this section, we will use another theorem for the lower bounds of
linear forms in logarithms which differs from that in above section
and the Baker-Davenport reduction method to deal with the remaining
cases. We recall the following result is due to Matveev
\cite{Matveev:2000}.

\begin{lemma}\label{lem:Matveev}
Denote by $\alpha\sb{1},\ldots,\alpha\sb{j}$ algebraic numbers, not
$0$ or $1$, by $\log\alpha\sb{1},$ $\ldots,\log\alpha\sb{j}$
determinations of their logarithms, by $D$ the degree over $\QQ$ of
the number field $\mathbb K =
\QQ(\alpha\sb{1},\ldots,\alpha\sb{j})$, and by
$b\sb{1},\ldots,b\sb{j}$ rational integers. Define
$B=\max\{|b\sb{1}|,\ldots,|b\sb{j}|\}$, and $A\sb{i}= \max\{D
h(\alpha\sb{i}),|\log\alpha\sb{i}|, 0.16\}$ ($1\le i\le j$), where
$h(\alpha)$ denotes the absolute logarithmic Weil height of
$\alpha$. Assume that the number
$$
\Lambda=b\sb{1}\log\alpha\sb{1}+\cdots+b\sb{n}\log\alpha\sb{j}
$$
does not vanish; then
\[|\Lambda|\ge\exp\{-C(j,\varkappa) D^2 A\sb{1}\cdots A\sb{j} \log (e D) \log (e B)\},\]
where $\varkappa=1$ if $\mathbb K \subset \RR$ and $\varkappa =2$
otherwise and
\[C(l,\varkappa)=\min \left\{ \frac 1{\varkappa} \left( \frac 12 ej \right) ^{\varkappa}
30^{j+3} j^{3.5}, 2^{6j +20} \right\}. \]
\end{lemma}

Now, we apply the above lemma with $j=3$ and $\varkappa=1$ for
$$
\Lambda=l\log \alpha -m \log \beta + \log
\gamma.
$$
Here we take $$D=4,\; b_1=l,\; b_2=-m,\; b_3=1, \;
\alpha_1=\alpha,\;\alpha_2=\beta,\;\alpha_3=\frac{\sqrt{c}(y_2+x_2\sqrt{b})}{\sqrt{b}(z_0+x_0\sqrt{c})}.$$
From the computations done in the previous section, we put
 $$h(\alpha_1)=\frac{1}{2}\log\alpha,\,\,
h(\alpha_2)=\frac{1}{2}\log \beta.$$
We see that $\alpha_3$ is a zero of the polynomial
$$
b^2(c-1)^2X^4 -4b^2c(c-1)x_0x_2X^3- 2bc\left((b-1)(c-1)-2b(c-1)x_2^2-2c(b-1)x_0^2\right)X^2$$
$$ -4bc^2(b-1)x_0x_2X+c^2(b-1)^2.
$$
This implies
$$
h(\alpha_3) \le \frac{1}{4}\left(\log\left(b^2(c-1)^2\right)+4\log\frac{\max\{|\sqrt{c}( y_2 \pm x_2\sqrt{b})|\}}{\min\{|\sqrt{b}(z_0\pm x_0\sqrt{c})|\}}\right)
$$
$$
=\frac{1}{4}\left(\log\left(b^2(c-1)^2\right)+4\log\frac{\sqrt{c}(1+\sqrt{b})}{\sqrt{b}(-t+r\sqrt{c})}\right)
$$
$$
\le \frac {1}{4}\left(\log\left(\frac{2^4 c^4(1+\sqrt{b})^4}{(c-1)^2} \right)\right)\le \log (2bc).
$$
Therefore, we take $$A_1 =2\log\alpha,\;\; A_2 =2\log \beta,\;\; A_3= 4\log (2bc).$$
Using Matveev's result we have
 \begin{equation}\label{eq:Matveev-lowerbound}
    \log|\Lambda|>-1.3901\cdot 10^{11}\cdot 16 \cdot {\log\alpha}\cdot {\log\beta} \cdot \log(2bc)
    \cdot \log (4e)\cdot \log (2el).
\end{equation}
By Lemma~\ref{lem:lambda}, we know that
$
\log\betrag{\Lambda} <-2m\log \beta.
$ It is easy to show that $m\log \beta > 0.5 l \log \alpha$. Combining the two bounds for $\log\betrag{\Lambda}$, we get
$$
\frac{l}{\log(2el)}<5.4\cdot 10^{12} \cdot \log\beta\cdot \log (2bc) < 5.4 \cdot 10^{12} \cdot \log^2 (2c^2).
$$
As $c=c_k^{(\pm )}\le c_k^{(+)}, k\le 7$ and $r\le 68962(k+4)+115$, we have $c<(2r)^{14}<3.44\cdot 10^{86}$. The above
inequality gives $l<2.2\cdot 10^{17}.$

In order to deal with the remaining cases, we will use a Diophantine approximation algorithm called the
Baker-Davenport reduction method. The following lemma is a slight modification of the original version of Baker-Davenport reduction method. (See \cite[Lemma 5a]{Dujella-Pethoe:1998}).

\begin{lemma}\label{lem:Baker-Davenport}
Assume that $M$ is a positive integer. Let $p/q$ be the convergent
of the continued fraction expansion of $\kappa$ such that $q > 6M$
and let
$$\eta=\left\| \mu q \right\| - M \cdot \| \kappa q\|,$$
where $\parallel\cdot\parallel$ denotes the distance from the
nearest integer. If $\eta > 0$, then there is no solution of the
inequality
$$0 < l\kappa -m+\mu < AB^{-l}$$
in integers $l$ and $m$ with
$$
\frac{\log\left(Aq/\eta\right)}{\log B}\leq l \leq M.
$$
\end{lemma}

We apply Lemma~\ref{lem:Baker-Davenport} to $\Lambda$ given by
\eqref{eq:def-Lambda} with
$$
\kappa=\frac{\log\alpha}{\log\beta},\quad
\mu=\frac{\log\gamma}{\log\beta},\quad A=
1,\quad B=\alpha,\;\; \mbox{ and }\;\; M =2.2\cdot 10^{17}.
$$

 The program was developed in PARI/GP running with $200$ digits precision. For the computations, if the first convergent such that $q > 6M$ does not satisfy the condition $\eta > 0$, then we use the next convergent until we find the one that satisfies the conditions. In $11$ hours, all the computations were done (using an Intel i7 4960HQ CPU). In all cases, we obtained $l\leq 42$. From Lemma~\ref{lem:mr}, we have $m\ge r-1$. By $m\le l$, we get $r\le 1+m < l +1 \le 43$. The second running provided $l\le 9$. We checked all cases and found no solution to $P_{l}=Q_{m} $, for $m\ge2$.  Then, we have.
\begin{proposition}\label{pro:2}
 For a triple $\{1,b,c_{k}^{(\pm)}\},\,(1\le k\le 7)$,  if $r\le68962(k+4)+115$, then equation $P_l=Q_m$ has no solution $(l,m)$ satisfying $m>1$.
\end{proposition}

Combining Proposition~\ref{pro:1} and Proposition~\ref{pro:2}, we deduce that if $P_{l} = Q_m$ has a positive integer solution $(l,m)$, then $m\le 1$. In fact, a solution comes from $m=1, l=k+1, z_0 = t$, which implies that $x=T_{k+1}\pm U_{k+1},y=\pm T_{k+1}+bU_{k+1}$. Thus, we get
$$d=x^2-1 =  (T_{k+1} \pm U_{k+1})^2 -1 = \pm 2 T_{k+1} U_{k+1} + r^2 U_{k+1}^2 =c_{k+1}^{(\pm)}.
$$
This completes the proof of our Theorem.


\end{document}